\documentclass{amsart}

\usepackage[utf8]{inputenc}
\usepackage{amsthm}
\usepackage{amsmath}
\usepackage{amssymb}
\usepackage{pgfplots}
\usepackage{mathrsfs}
\usepackage{tikz-cd}
\usepackage[cal=boondox]{mathalfa}
\usepackage{mathtools}
\usepackage{xcolor}
\usepackage{enumitem}
\usepackage{caption}
\usepackage{float}
\usepackage{natbib}
\usepackage{eucal}
\usepackage{lipsum}
\usepackage{pdflscape}
\usepackage{cancel}
\usepackage{amsfonts}
\usepackage{url}
\usepackage{amssymb}
\usepackage{enumitem}
\theoremstyle{plain}
\newtheorem{theorem}{Theorem}[section]
\theoremstyle{definition}
\newtheorem{definition}[theorem]{Definition}

\newtheorem{corollary}[theorem]{Corollary}
\newtheorem{lemma}[theorem]{Lemma}

\newtheorem{remark}{Remark}

\usepackage{multicol}
\usepackage{tikz}
\usetikzlibrary{positioning}
\usepackage[]{hyperref}
\tikzset{sgplattice/.style={inner sep=1pt,norm/.style={black!50!black},char/.style={black!50!black},
  lin/.style={black!50}},cnj/.style={black!50,yshift=-2.5pt,left=-1pt of #1,scale=0.5,fill=white}}
\newcommand{\hsep}{\hspace{3pt}}

\newcommand{\vsep}{

\vspace{7pt}

}

\newcommand{\defeq}{\coloneqq}

\usepackage{orcidlink}
\usepackage{lipsum}
\usepackage{subfiles} 

\title{Galois subspaces for compact Riemann surfaces of genus 2}
\author{Juan-Pablo Llerena-C\'ordova \orcidlink{0009-0005-0232-5639}}
\thanks{\noindent Department of Mathematics and Statistics, University of Ottawa\\
150 Louis-Pasteur Pvt, Ottawa, ON, Canada K1N 6N5\\
jller032@uottawa.ca}
\begin{document}

\maketitle

\begin{abstract}
	Let $X$ be a compact Riemann surface of genus 2 and $D$ a very ample divisor with $\phi_D$ its associated embedding into $\mathbb{P}^{n}$. We consider the set $G_{X,D}$ of linear subspaces $W$ of $\mathbb{P}^n$ of codimension $2$ with projection $\pi_W$ such that $f_W = \pi_W \circ \phi_D$ is Galois, i.e. $f_W^*k(\mathbb{P}^1) \subseteq k(X)$ is a Galois extension. It is known that $G_{X,D}$ is isomorphic to a disjoint union of projective spaces. In this article, we calculate the dimension of projective spaces in the decomposition of $G_{X,D}$, when $D$ is induced by a subgroup of $\mathrm{Aut}(X)$.
\end{abstract} 

\section*{Introduction}

Let $X$ be a compact Riemann surface of genus $g$, $k(X)$ the function field of $X$, $D$ a very ample divisor on $X$, and $\phi_D:X \rightarrow \mathbb{P}^N$ the embedding induced by $D$. If $W\in \mathbb{G}(N-2, N)$, we denote by $\pi_W:\mathbb{P}^N \rightarrow \mathbb{P}^1$ the projection away from $W$.
\begin{definition}
	We say that $\phi_D$ is a \textit{Galois embedding} if there exists a subspace $W\in \mathbb{G}(N - 2, N)$, such that $f_W = \pi_W \circ \phi_D$ is a Galois cover; that is, the field extension $f_W^*k(\mathbb{P}^1)\subseteq k(X)$ is Galois. In this case, we say that $W$ is a \textit{Galois subspace} for $D$. Moreover, if $W$ is disjoint from $\phi_D(X)$, then we say that $W$ is a \textit{disjoint Galois subspace}.
\end{definition}
The general concept of Galois embeddings and Galois subspaces was introduced by Yoshihara in \cite{First_Yoshihara}. The particular case when $W$ is a point was introduced by Fukasawa in \cite{MR2290435}. In recent years, multiple authors have expanded on the idea of Galois subspaces, \textit{e.g.}, \cite{MR4283171}, \cite{Random_2}, \cite{Random_3}, \cite{MR3844636}, and \cite{Auffarth_2023}. 
 We build on the work of Auffarth and Rahausen \cite{MR4283171}, in which the authors study the following set
\[
	G_{X,D} \defeq \{W\in \mathbb{G}(N - 2,  N); \hsep \pi_W \circ \phi_D \text{ is Galois}\}.
\]
They prove (see \cite[Theorem 1.1]{MR4283171}) that if $g \geq 2$, then
\begin{equation}\label{decomposition}
	G_{X,D} \cong \bigsqcup_{\substack{H \leq \text{Aut}(X)\\ X/H \cong \mathbb{P}^1}} \mathcal{X}_{D, \pi_H},
\end{equation}
where $\pi_H$ denotes the canonical projection $X \rightarrow X/H$, and $\mathcal{X}_{D, \pi_H}$ is isomorphic to a projective space of dimension $\ell(D - D_{\pi_H}) - 1$, where $\ell(D - D_{\pi_H}) \defeq \text{dim}H^0(X,\mathcal{O}_X(D - D_{\pi_H}))$. In particular, the disjoint Galois subspaces are in the $0$-dimensional components of the decomposition of Equation \eqref{decomposition} (see \cite[Remark 3.1]{MR4283171}).
\vsep
The set $G_{X,D}$ has been studied when $X$ has genus $0$ or $1$ in \cite{MR4283171}, and \cite{MR4198195}, respectively. In this article, we will study the set $G_{X,D}$ when $g = 2$, using the decomposition of Equation \eqref{decomposition}. Because $\mathcal{X}_{D, \pi_H}$ is a projective space of dimension $\ell(D - D_{\pi_H}) - 1$ and $\text{Aut}(X)$ is finite, to study the components of the decomposition of $G_{X,D}$, we only need to focus on calculating the dimension of each projective space. This is possible because the isomorphism class of $\text{Aut}(X)$ is known for genus $2$ (see \cite{zbMATH01251322} and \cite{MR1090743}). 
\vsep
We mainly focus on disjoint Galois subspaces, which lie in the $0$-dimensional components of the decomposition of Equation \eqref{decomposition}. Moreover, $G_{X,D}$ has a $0$-dimensional component if and only if there exists a subgroup $H\leq \text{Aut}(X)$ with $X/H \cong \mathbb{P}^1$ such that $D \sim D_{H}$, where $D_H\defeq D_{H,p} \defeq \sum_{h\in H} h[p]$ for some $p\in X$ (see Lemma \ref{DisjointGalois}). Note that, because $X/H \cong \mathbb{P}^1$, the definition of $D_{H}$ is independent of the choice of $p$, up to linear equivalence (see Lemma \ref{lemma_independence}).
\vsep
The main result of this article describes the components of the decomposition of $G_{X,D}$ in Equation \eqref{decomposition} when $D$ is induced by a subgroup of $\text{Aut}(X)$, to guarantee the existence of a $0$-dimensional component.
\begin{theorem}\label{Main_theorem}
	Let $X$ be a compact Riemann surface of genus $2$ and let $H \leq G = \text{Aut}(X)$. The following assertions hold:
	\begin{enumerate}
		\item\label{Main_theorem_1} Assume $G\not\cong C_{10}$. Two very ample divisors induced by subgroups of $G$ differ by a multiple of the canonical divisor.
		\item In $G_{X,D_H}$, there is a one-to-one correspondence between disjoint Galois subspaces and subgroups of order $|H|$.
	\end{enumerate}
\end{theorem}
The structure of the article is as follows:
\vsep
In Section 1, we prove some general results on compact Riemann surfaces of genus $g \geq 2$. We prove Theorem \ref{Main_theorem} in Section 2. Lastly, in Section $3$, we present a table describing the dimension of the components of $G_{X,D_H}$.
\begin{remark}
	Even though the techniques used to prove Theorem \ref{Main_theorem} hold for arbitrary genus, that is, Lemma \ref{First_lemma} and Lemma \ref{secondlemma}, we are unable to prove an analog of Theorem \ref{Main_theorem} for genus greater than $2$. The main reason is that for a Riemann surface $X$ of genus $2$, there exist many subgroups $H\leq \mathrm{Aut}(X)$ such that $X/H \cong \mathbb{P}^1$. This is not the case when the genus is greater than $2$. Although the structure of $\mathrm{Aut}(X)$ is well-known for low genus, it is not clear whether the techniques used in the proof of Theorem \ref{Main_theorem} can be used to give a similar result for higher genus.
\end{remark}
Finally, we note that Yoshihara and Fukasawa have a list of current open problems of Galois subspaces \cite{List_open_problems}, in which they also summarize what is currently known regarding each problem.
\section{Linear equivalence of divisors given by a subgroup}
In this section, we will prove the lemmas necessary to prove Theorem \ref{Main_theorem}. For this section, unless stated otherwise, $X$ will be a compact Riemann surface of genus $g\geq 2$, and $G = \text{Aut}(X)$.
\vsep
Recall that given a subgroup $H \leq G$ with $X/H \cong \mathbb P^1$, we denote by $D_H \defeq D_{H,p} \defeq \sum_{h\in H} h[p]$ the divisor induced by $H$, for some $p \in X$. This is independent of the choice of $p$, up to linear equivalence, by the following lemma.
\begin{lemma}\label{lemma_independence}
	Let $H \leq G$ be such that $X/H \cong \mathbb P^1$. Then, for any $p,q\in X$ we have that $D_{H,p} \sim D_{H,q}$. That is, $D_{H,p}$ is independent of the choice of the point $p$, up to linear equivalence.
\end{lemma}
\begin{proof}
	Let $p,q\in X$, and $\pi:X \rightarrow X/H$ be the natural projection. Because $X/H \cong \mathbb{P}^1$, we have that $\pi(p) \sim \pi(q)$ and therefore $\pi^{-1}(\pi(p)) \sim \pi^{-1}(\pi(q))$, which follows from \cite[Chapter V Lemma 2.3]{Mir95} and \cite[Chapter V Corollary 2.6]{Mir95}. Noting that $\pi^{-1}(\pi(p)) = D_{H,p}$ and $\pi^{-1}(\pi(q)) = D_{H,q}$, we can conclude that $D_{H,p} \sim D_{H,q}$.
\end{proof}
By the previous lemma, we can denote $D_H$ without specifying the point $p$. Now, we have the following lemma.
\begin{lemma}\label{First_lemma}
	If $H, N,L \leq G$ with $L \subseteq N \cap H$, $X/L \cong \mathbb{P}^1$, and $|H| - |N| = |L|$, then $D_H - D_N \sim D_L$.
\end{lemma}
\begin{proof}
	Because $L \subseteq H$, we can consider the set of right cosets\footnote{Note that we are not assuming that $L$ is a normal subgroup of $H$ or $N$. Nonetheless, we can still consider the right cosets of a subgroup} of $L$ in $H$. So, by fixing  a set of representatives of the right cosets $H_L$, we have that
	\[
		D_{H,p} = \sum_{h\in H} h[p] = \sum_{h\in H_L} (Lh)[p] = \sum_{h\in H_L} L(h[p]) = \sum_{h\in H_L} D_{L,h[p]}.
	\]
	Doing the same process with $D_{N,p}$, and denoting by $N_L$ a set of representatives of right cosets of $L$ in $N$, we get that
	\begin{equation}\label{eqq1}
		D_{H,p} - D_{N,p} - D_{L,p} = \sum_{h\in H_L} D_{L,h[p]} - \sum_{n\in N_L} D_{L,n[p]} - D_{L,p},
	\end{equation}
	Because $X/L \cong \mathbb{P}^1$ and $D_{L,p}= \pi^{-1}(\pi(p))$, where $\pi:X\rightarrow X/L$ denotes the natural projection, as in the proof of Lemma \ref{lemma_independence}, we have that $D_{L,q} \sim D_{L,p}$ for all $p,q\in X$. Furthermore, because  $|H| - |N| = |L|$ we have that $|H_L| - |N_L| = 1$. Therefore, rearranging Equation \eqref{eqq1}, we can conclude that $D_{H,p}-D_{N,p} \sim D_{L,p}$.
\end{proof}

Because every compact Riemann surface $X$ of genus $2$ is hyperelliptic, we have that there exists a unique hyperelliptic involution $\sigma$ such that $X/\left<\sigma\right> \cong \mathbb{P}^1$, $\left<\sigma\right> \leq Z(G)$, and $D_{\left<\sigma\right>} \sim K$, where $K$ is a canonical divisor. So, when $X$ has genus $2$, we have the following corollary.

\begin{corollary}\label{corollary_cool}
	Let $X$ be a compact Riemann surface of genus 2 and $\sigma$ its hyperelliptic involution. If $H,N\leq G$ with $\left<\sigma\right> \subseteq N \cap H$, and $|H| - |N| = 2$, then $D_H - D_N \sim K$. Moreover, $\ell(D_H - D_N) = 2$.
\end{corollary}
\begin{proof}
	We can use the previous lemma, with $L = \left<\sigma\right>$, and the fact that if $D_H - D_N \sim K$, then
	\[
		\ell(D_H - D_N) = \ell(K) = g = 2
	\]
	to deduce the corollary.
\end{proof}
Returning to the general case, we have the following lemma.
\begin{lemma}\label{secondlemma}
	Let $X$ be a compact Riemann surface of genus $g \geq 2$ with $G = \text{Aut}(X)$. If $H,N,L \leq G$ with $L \subseteq N \cap H$, $X/L \cong \mathbb{P}^1$, and $|H| = |N|$, then $D_H \sim D_N$.
\end{lemma}
\begin{proof}
	This can be proven in the same way as in Lemma \ref{First_lemma}.
\end{proof}
Now, if $X$ has genus $2$, then we have an equivalent characterization of a very ample divisor.
\begin{lemma}\label{very_ample_divisor}
	Let $X$ be of genus $2$ and $D$ a divisor on $X$. We have that $D$ is a very ample divisor if and only if $\deg(D) \geq 5$.
\end{lemma}
\begin{proof}
	The ``if'' part of the lemma follows from the Riemann-Roch formula and the fact that $\deg(D) > \deg(K) + 2$ (see \cite[Chapter VII Proposition 1.2]{Mir95}). The ``only if'' part of the lemma is a consequence of the fact that there is no immersion from $X$ into $\mathbb{P}^{n}$ with $n < 3$; the case $n = 1$ follows from Hurwitz's Formula and the case $n = 2$ contradicts the Genus-Degree Formula.
\end{proof}
Finally, we have the following lemma.
\begin{lemma}\label{Proposition_order}
	If $X$ has genus $2$ and $D_1,D_2$ are two divisors on $X$ with $\deg(D_1) - \deg(D_2) > 2$, then $\ell(D_1 - D_2) = \deg(D_1) - \deg(D_2) - 1$.
\end{lemma}
\begin{proof}
	This can be proven using the Riemann-Roch formula and the fact that $\deg(D_1) - \deg(D_2) > \deg(K)$.
\end{proof}
\section{Proof of main result}
We will use the notation of the previous section. As mentioned in the introduction, to study the structure of the $0$-dimensional component of $G_{X,D}$, we only need to consider the case where $D$ is a very ample divisor induced by a subgroup $H \leq \text{Aut}(X)$. Indeed, this follows from the following lemma.
\begin{lemma}\label{DisjointGalois}
	Let $X$ be a compact Riemann surface of genus $g \geq 2$ and $G = \text{Aut}(X)$. Then, $G_{X,D}$ has a $0$-dimensional component if and only if $D \sim D_H$ for some $H \leq G$ with $X/H \cong \mathbb{P}^1$.
\end{lemma}
\begin{proof}
	Assume that $G_{X,D}$ has a $0$-dimensional component, then by the decomposition of Equation \eqref{decomposition} we have that there exists a subgroup $H \leq G$ with $X/H \cong \mathbb{P}^1$ such that $\ell(D - D_H) = 1$. Therefore, $D \sim D_H$. Conversely, if there exists a subgroup $H \leq G$ with $X/H \cong \mathbb{P}^1$ such that $D \sim D_H$, then $\ell(D - D_H) = \ell(D_H - D_H) = 1$, so $G_{X,D_H}$ has a $0$-dimensional component.
\end{proof}
So, if we want to study the disjoint Galois subspaces we can assume that $D$ is induced by a subgroup $H \leq G$ with $X/H \cong \mathbb{P}^1$. Thus, we only need to calculate $\ell(D_H - D_N)$ where $N \leq \text{Aut}(X)$ and $X/N\cong \mathbb{P}^1$. Using Lemma \ref{very_ample_divisor}, we know that $D_H$ is very ample if and only if $\text{deg}(D_H) \geq 5$. Furthermore, we have that $\text{deg}(D_H) = |H|$. Thus, we only consider $D_H$ when $|H| \geq 5$.
\vsep
It is known that if $X$ is a compact Riemann surface of genus $2$, then $\text{Aut}(X)$ is isomorphic to one of the following groups (see \cite[Table 1]{zbMATH01251322}):
\begin{equation}\label{automorphism_groups}
	C_2, C_2 \times C_2, D_4, C_{10}, D_6, C_3\rtimes D_4, \mathrm{GL}_2(\mathbb{F}_3),
\end{equation}
where $D_{n}$ denotes the dihedral group of order $2n$ (we note that \cite{zbMATH01251322} denotes by $D_{2n}$ the dihedral group of order $2n$). Furthermore, we use \cite{GroupNames} to obtain information on the lattice of subgroups of $G$.
\vsep
We see that when $\text{Aut}(X) \cong C_2$ or $\text{Aut}(X) \cong C_2 \times C_2$ there are no subgroups of order at least $5$, so $G_{X,D_H}$ is empty.
\vsep
Now, for $\text{Aut}(X) \cong D_4$ or $\text{Aut}(X) \cong C_{10}$, we have that $D_H$ if and only if $|H| \geq 5$ (Lemma \ref{very_ample_divisor}). In particular, if $H \ne N\leq \text{Aut}(X)$ then $|H| - |N| > 2$. Thus, we can use Lemma \ref{Proposition_order} to calculate $\ell(D_H - D_N)$.
\vsep
The first non-trivial case is when $G \cong D_6$. We can use \cite{GroupNames} to obtain the lattice of subgroups of $G$:\footnote{The left subscript indicates the number of subgroups isomorphic to the group in the node.}
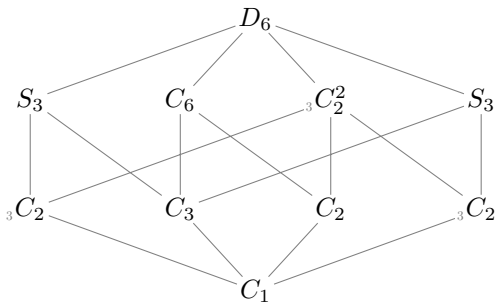
\begin{center}
	\begin{tikzpicture}[scale=1.0,sgplattice]
		\node[char] at (3.12,0) (1) {$C_1$};
		\node[char] at (4.12,1.09) (2) {$C_2$};
		\node at (0.125,1.09) (3) {$C_2$};
		\node at (6.12,1.09) (4) {$C_2$};
		\node[char] at (2.12,1.09) (5) {$C_3$};
		\node at (4.12,2.54) (6) {$C_2^2$};
		\node[norm] at (0.125,2.54) (7) {$S_3$};
		\node[norm] at (6.12,2.54) (8) {$S_3$};
		\node[char] at (2.12,2.54) (9) {$C_6$};
		\node[char] at (3.12,3.62) (10) {$D_6$};
		\draw[lin] (1)--(2) (1)--(3) (1)--(4) (1)--(5) (2)--(6) (3)--(6) (4)--(6)
		 (3)--(7) (5)--(7) (4)--(8) (5)--(8) (2)--(9) (5)--(9) (6)--(10)
		 (7)--(10) (8)--(10) (9)--(10);
		\node[cnj=3] {3};
		\node[cnj=4] {3};
		\node[cnj=6] {3};
	\end{tikzpicture}
\end{center}
If $H = D_6$ then $|H| - |N| > 2$ for any proper subgroup $N \leq D_6$, so we use Lemma \ref{Proposition_order} to calculate $\ell(D_H - D_N)$. Now, we see that the only subgroups of order at least $5$ are when $H \cong C_6$ or $H \cong S_3$. Again, we can use Lemma \ref{Proposition_order} to calculate $\ell(D_H - D_N)$ when $|N| < 4$. So, we will only focus on the cases when $|N| = 6$ or $|N| = 4$.
\vsep
Using the classification made by Broughton in \cite{MR1090743}, we can see that the only possible subgroup whose quotient is not isomorphic to $\mathbb{P}^1$ is of order $2$ (see \cite[Table 4]{MR1090743}). Therefore if $L \leq G$ with $|L| \ne 2$, then $X/L \cong \mathbb{P}^1$. Also, because the hyperelliptic involution is the only automorphism of order 2 whose quotient is $\mathbb{P}^1$, we have that if $L \leq G$ with $|L| = 2$ and $L \ne \left<\sigma \right>$ then $X/L \not \cong \mathbb{P}^1$. Lastly, recall that $\left<\sigma\right>$ is contained in the center of $G$.
\vsep
Therefore, if we want to calculate $\ell(D_H - D_N)$, when $H \cong S_3$ (the right one) and $N\cong C_6$, we can make the following observation: $H\cap N \cong C_3$ and $X/(H\cap N) \cong \mathbb{P}^1$, so we can use Lemma \ref{secondlemma}, and see that $D_H \sim D_N$.
\vsep
Another case is when $H \cong S_3$ (the left one) and $N \cong C_2 \times C_2$. In this case, we can make the following drawing
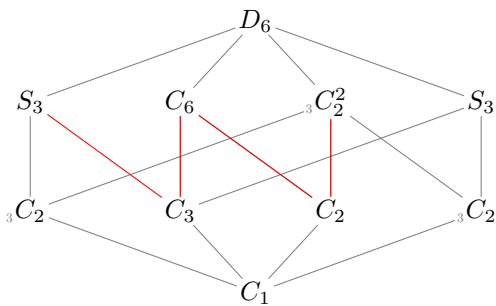
\begin{center}
	\begin{tikzpicture}[scale=1.0,sgplattice]
		\node[char] at (3.12,0) (1) {$C_1$};
		\node[char] at (4.12,1.09) (2) {$C_2$};
		\node at (0.125,1.09) (3) {$C_2$};
		\node at (6.12,1.09) (4) {$C_2$};
		\node[char] at (2.12,1.09) (5) {$C_3$};
		\node at (4.12,2.54) (6) {$C_2^2$};
		\node[norm] at (0.125,2.54) (7) {$S_3$};
		\node[norm] at (6.12,2.54) (8) {$S_3$};
		\node[char] at (2.12,2.54) (9) {$C_6$};
		\node[char] at (3.12,3.62) (10) {$D_6$};
		\draw[lin] (1)--(2) (1)--(3) (1)--(4) (1)--(5) (2)--(6) (3)--(6) (4)--(6)
		 (3)--(7) (4)--(8) (5)--(8) (2)--(9) (5)--(9) (6)--(10)
		 (7)--(10) (8)--(10) (9)--(10);
		\draw[lin, red] (7)--(5) (5)--(9) (9)--(2) (2)--(6);
		\node[cnj=3] {3};
		\node[cnj=4] {3};
		\node[cnj=6] {3}; 
	\end{tikzpicture}
\end{center}
As mentioned in the previous case, we have that $D_H \sim D_{C_6}$ and because $C_6 \cap N = Z(G)$, we can use Corollary \ref{corollary_cool} and see that $D_{C_6} - D_{N} \sim K$. Therefore, $D_H - D_N \sim K$. So, $\ell(D_H - D_N) = 2$.
\vsep
This ``Zig-Zag'' process can be done to calculate $\ell(D_H - D_N)$ for all necessary cases.
\vsep
Another example is when $G \cong \mathrm{GL}_2(\mathbb{F}_3)$. Again, we can use \cite{GroupNames} to obtain the lattice of $G$:
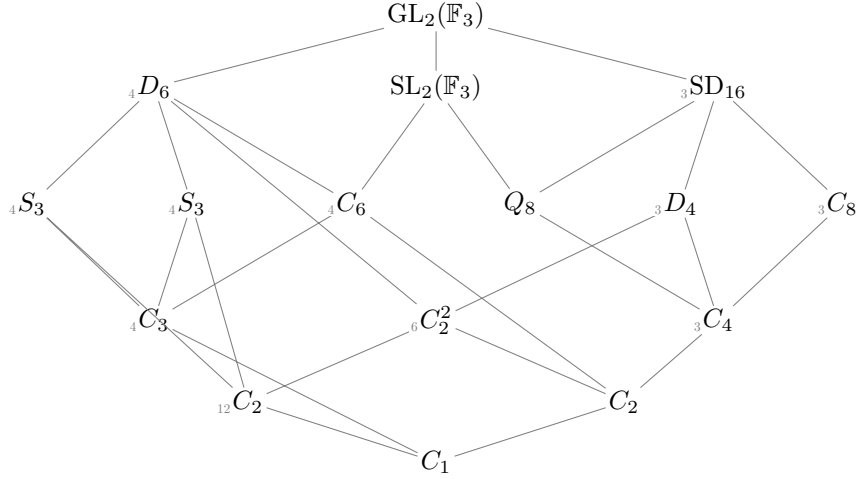
\begin{center}
	\begin{tikzpicture}[scale=1.0,sgplattice]
  \node[char] at (5.5,0) (1) {$C_1$};
  \node[char] at (8,0.803) (2) {$C_2$};
  \node at (3,0.803) (3) {$C_2$};
  \node at (1.75,1.89) (4) {$C_3$};
  \node at (9.25,1.89) (5) {$C_4$};
  \node at (5.5,1.89) (6) {$C_2^2$};
  \node at (0.125,3.44) (7) {$S_3$};
  \node at (4.38,3.44) (8) {$C_6$};
  \node at (2.25,3.44) (9) {$S_3$};
  \node[char] at (6.62,3.44) (10) {$Q_8$};
  \node at (8.75,3.44) (11) {$D_4$};
  \node at (10.9,3.44) (12) {$C_8$};
  \node at (1.75,4.99) (13) {$D_6$};
  \node at (9.25,4.99) (14) {${\rm SD}_{16}$};
  \node[char] at (5.5,4.99) (15) {${\rm SL}_2({\mathbb F}_3)$};
  \node[char] at (5.5,5.95) (16) {${\rm GL}_2({\mathbb F}_3)$};
  \draw[lin] (1)--(2) (1)--(3) (1)--(4) (2)--(5) (2)--(6) (3)--(6) (3)--(7)
	 (4)--(7) (2)--(8) (4)--(8) (3)--(9) (4)--(9) (5)--(10) (5)--(11)
	 (6)--(11) (5)--(12) (6)--(13) (7)--(13) (8)--(13) (9)--(13) (12)--(14)
	 (10)--(14) (11)--(14) (8)--(15) (10)--(15) (13)--(16) (14)--(16) (15)--(16);
  \node[cnj=3] {12};
  \node[cnj=4] {4};
  \node[cnj=5] {3};
  \node[cnj=6] {6};
  \node[cnj=7] {4};
  \node[cnj=8] {4};
  \node[cnj=9] {4};
  \node[cnj=11] {3};
  \node[cnj=12] {3};
  \node[cnj=13] {4};
  \node[cnj=14] {3};
\end{tikzpicture}
\end{center}

For example, if we want to calculate $\ell(D_H - D_N)$ when $H \cong C_8$ and $N \cong S_3$ (the leftmost one), then we can see that $C_8 \cap C_6 = Z(G)$, so $X/(C_8\cap C_6) \cong \mathbb{P}^1$ which lets us make the following diagram:
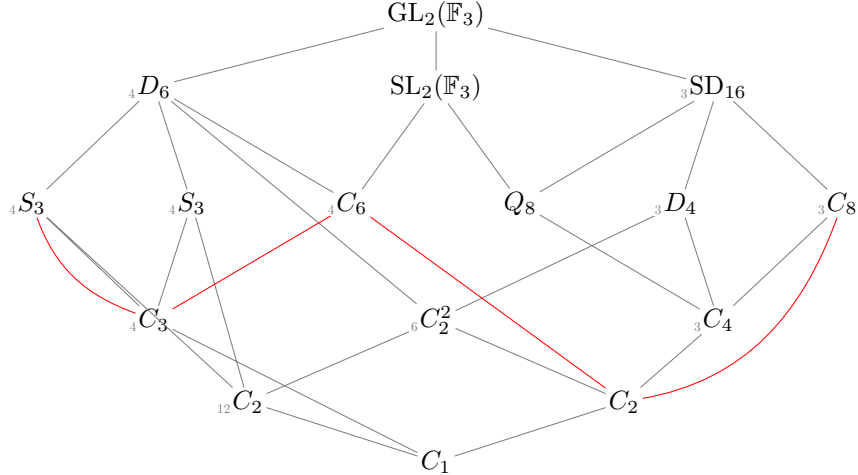
\begin{center}
	\begin{tikzpicture}[scale=1.0,sgplattice]
  \node[char] at (5.5,0) (1) {$C_1$};
  \node[char] at (8,0.803) (2) {$C_2$};
  \node at (3,0.803) (3) {$C_2$};
  \node at (1.75,1.89) (4) {$C_3$};
  \node at (9.25,1.89) (5) {$C_4$};
  \node at (5.5,1.89) (6) {$C_2^2$};
  \node at (0.125,3.44) (7) {$S_3$};
  \node at (4.38,3.44) (8) {$C_6$};
  \node at (2.25,3.44) (9) {$S_3$};
  \node[char] at (6.62,3.44) (10) {$Q_8$};
  \node at (8.75,3.44) (11) {$D_4$};
  \node at (10.9,3.44) (12) {$C_8$};
  \node at (1.75,4.99) (13) {$D_6$};
  \node at (9.25,4.99) (14) {${\rm SD}_{16}$};
  \node[char] at (5.5,4.99) (15) {${\rm SL}_2({\mathbb F}_3)$};
  \node[char] at (5.5,5.95) (16) {${\rm GL}_2({\mathbb F}_3)$};
  \draw[lin] (1)--(2) (1)--(3) (1)--(4) (2)--(5) (2)--(6) (3)--(6) (3)--(7)
	 (3)--(9) (4)--(9) (5)--(10) (5)--(11) (4)--(7)
	 (6)--(11) (5)--(12) (6)--(13) (7)--(13) (8)--(13) (9)--(13) (12)--(14)
	 (10)--(14) (11)--(14) (8)--(15) (10)--(15) (13)--(16) (14)--(16) (15)--(16);
  \draw[lin, red] (4)--(8) (8)--(2);
  \draw[lin, red] (2) to[out=10, in=-110] (12);
  \draw[lin, red] (7) to[out=-70,in=-200] (4);  
  \node[cnj=3] {12};
  \node[cnj=4] {4};
  \node[cnj=5] {3};
  \node[cnj=6] {6};
  \node[cnj=7] {4};
  \node[cnj=8] {4};
  \node[cnj=9] {4};
  \node[cnj=11] {3};
  \node[cnj=12] {3};
  \node[cnj=13] {4};
  \node[cnj=14] {3};
	\end{tikzpicture}
\end{center}

We can conclude that $D_H - D_N \sim K$ and that $\ell(D_H - D_N) = 2$. Again, this ``Zig-Zag'' process can be done to calculate $\ell(D_H - D_N)$ for all necessary cases.
\vsep
Doing this process with all possible $G$, we can prove Theorem \ref{Main_theorem}.
\section{Components of $G_{X,D_H}$}
Doing the process of Section $2$, with every possible automorphism group, we obtain the following table listing the dimensions of the components of $G_{X,D_H}$. We define
\[
	\mathcal{D}_n^H \coloneqq \# \{N\leq \text{Aut}(X); \hsep \ell(D_H - D_N) - 1 = n\}.
\]
\begin{center}
   \begin{table}[H]
	\begin{tabular}{c|c|c|c}
		$|H|$ & $\mathcal{D}_0^H$ & $\mathcal{D}_2^H$ & $\mathcal{D}_{4}^H$\\\hline
		$8$ & 1 & 3 & 1\\
	\end{tabular}
	\caption{Dimension components of $G_{X,D}$ with $\text{Aut}(X) \cong D_4$}
\end{table} 
\end{center}
\begin{center}
   \begin{table}[H]
	\begin{tabular}{c|c|c|c|c|c}
		$|H|$ & $\mathcal{D}_{-1}^H$ & $\mathcal{D}_{0}^H$ & $\mathcal{D}_{1}^H$& $\mathcal{D}_{2}^H$ & $\mathcal{D}_{6}^H$\\\hline
		$5$ & 1 & 1 & 1 & 0 & 0 \\
		$10$ & 0 & 1 & 0 & 1 & 1
	\end{tabular}
	\caption{Dimension components of $G_{X,D}$ with $\text{Aut}(X) \cong C_{10}$}
\end{table} 
\end{center}
\begin{center}
	\begin{table}[H]
	\begin{tabular}{c|c|c|c|c|c|c|c|c|c}
		$|H|$ & $\mathcal{D}_{-1}^H$ & $\mathcal{D}_{0}^H$&$\mathcal{D}_{1}^H$& $\mathcal{D}_{2}^H$ & $\mathcal{D}_{4}^H$ & $\mathcal{D}_{6}^H$ & $\mathcal{D}_{7}^H$ & $\mathcal{D}_{8}^H$\\\hline
		$6$ & 1 & 3 & 4 & 1 & 0 & 0 & 0 & 0 \\
		$12$ & 0 & 1 & 0 & 0 & 3 & 3 & 1 & 1 
	\end{tabular}
	\caption{Dimension components of $G_{X,D}$ with $\text{Aut}(X) \cong D_6$}
\end{table}
\end{center}
\begin{figure}[H]
\centering
\rotatebox{90}{%
  \begin{minipage}{\textheight}
    \centering

    \begin{minipage}{\linewidth}
      \centering
      \begin{tabular}{c|c|c|c|c|c|c|c|c|c|c|c|c|c|c|c|c}
        $|H|$ & $\mathcal{D}_{-1}^H$ & $\mathcal{D}_{0}^H$ & $\mathcal{D}_{1}^H$ & $\mathcal{D}_{2}^H$ & $\mathcal{D}_{3}^H$ & $\mathcal{D}_{4}^H$ & $\mathcal{D}_{5}^H$ & $\mathcal{D}_{6}^H$ & $\mathcal{D}_{7}^H$ & $\mathcal{D}_{8}^H$ & $\mathcal{D}_{10}^H$ & $\mathcal{D}_{14}^H$ & $\mathcal{D}_{16}^H$ & $\mathcal{D}_{18}^H$ & $\mathcal{D}_{19}^H$ & $\mathcal{D}_{20}^H$\\\hline
        $6$  & 7 & 5 & 8 & 1 & 0 & 0 & 0 & 0 & 0 & 0 & 0 & 0 & 0 & 0 & 0 & 0\\
        $8$  & 4 & 3 & 5 & 7 & 1 & 1 & 0 & 0 & 0 & 0 & 0 & 0 & 0 & 0 & 0 & 0\\
        $12$ & 1 & 3 & 0 & 3 & 0 & 0 & 5 & 7 & 1 & 1 & 0 & 0 & 0 & 0 & 0 & 0\\
        $24$ & 0 & 1 & 0 & 0 & 0 & 0 & 0 & 0 & 0 & 0 & 3 & 3 & 5 & 7 & 1 & 1\\
      \end{tabular}
      \captionof{table}{Dimension components of $G_{X,D}$ with $\text{Aut}(X) \cong C_3 \rtimes D_4$}
    \end{minipage}

    \vspace{1.5em}
	\hspace{-15em}
    \begin{minipage}{\linewidth}
	
      \centering
	  \scalebox{0.90}{
      \begin{tabular}{c|c|c|c|c|c|c|c|c|c|c|c|c|c|c|c|c|c|c|c|c|c|c|c|c|c|c|c}
        $|H|$ & $\mathcal{D}_{-1}^H$ & $\mathcal{D}_{0}^H$ & $\mathcal{D}_{1}^H$ & $\mathcal{D}_{2}^H$ & $\mathcal{D}_{3}^H$ & $\mathcal{D}_{4}^H$ & $\mathcal{D}_{5}^H$ & $\mathcal{D}_{6}^H$ & $\mathcal{D}_{7}^H$ & $\mathcal{D}_{8}^H$ & $\mathcal{D}_{10}^H$ & $\mathcal{D}_{11}^H$ & $\mathcal{D}_{12}^H$ & $\mathcal{D}_{14}^H$ & $\mathcal{D}_{16}^H$ & $\mathcal{D}_{18}^H$ & $\mathcal{D}_{19}^H$ & $\mathcal{D}_{20}^H$ & $\mathcal{D}_{22}^H$ & $\mathcal{D}_{30}^H$ & $\mathcal{D}_{34}^H$ & $\mathcal{D}_{38}^H$ & $\mathcal{D}_{40}^H$ & $\mathcal{D}_{42}^H$ & $\mathcal{D}_{43}^H$ & $\mathcal{D}_{44}^H$\\\hline
        $6$  & 16 & 12 & 13 & 1 & 0 & 0 & 0 & 0 & 0 & 0 & 0 & 0 & 0 & 0 & 0 & 0 & 0 & 0 & 0 & 0 & 0 & 0 & 0 & 0 & 0 & 0\\
        $8$  & 9 & 7 & 12 & 9 & 4 & 1 & 0 & 0 & 0 & 0 & 0 & 0 & 0 & 0 & 0 & 0 & 0 & 0 & 0 & 0 & 0 & 0 & 0 & 0 & 0 & 0\\
        $12$ & 5 & 4 & 0 & 7 & 0 & 12 & 0 & 9 & 4 & 1 & 0 & 0 & 0 & 0 & 0 & 0 & 0 & 0 & 0 & 0 & 0 & 0 & 0 & 0 & 0 & 0\\
		$16$ & 2 & 3 & 0 & 4 & 0 & 0 & 0 & 7 & 0 & 12 & 9 & 4 & 1 & 0 & 0 & 0 & 0 & 0 & 0 & 0 & 0 & 0 & 0 & 0 & 0 & 0\\
        $24$ & 1 & 1 & 0 & 0 & 0 & 0 & 0 & 3 & 0 & 0 & 4 & 0 & 0 & 7 & 12 & 9 & 4 & 1 & 0 & 0 & 0 & 0 & 0 & 0 & 0 & 0\\
		$48$ & 0 & 1 & 0 & 0 & 0 & 0 & 0 & 0 & 0 & 0 & 0 & 0 & 0 & 0 & 0 & 0 & 0 & 0 & 1 & 3 & 4 & 7 & 12 & 9 & 4 & 1\\
      \end{tabular}
	  }
      \captionof{table}{Dimension components of $G_{X,D}$ with $\text{Aut}(X) \cong \mathrm{GL}_2(\mathbb{F}_3)$}
    \end{minipage}

  \end{minipage}
} 
\end{figure}
\textit{Acknowledgments: The author would like to thank Robert Auffarth for his encouragement and help with this article. The author would also like to thank Antonio Lei for feedback on earlier drafts. The author was supported by Agencia Nacional de Investigaci\'on y Desarrollo-Subdirección de Capital Humano scholarship Mag\'ister Nacional 2022 N° 22221372, during part of the development of this article.}
\bibliographystyle{plain}
\bibliography{References}
\end{document}